\documentclass[11pt,a4paper]{amsart}
\usepackage{amssymb,amsmath,epsfig,graphics,mathrsfs}

\usepackage{fancyhdr}
\usepackage{hyperref}
\hypersetup{
 colorlinks   = true,
 urlcolor     = blue,
 linkcolor    = blue,
 citecolor   = red ,
 bookmarksopen=true
}

\usepackage{amsmath}
\usepackage{amsfonts}
\usepackage{amssymb}
\usepackage{amsthm}
\usepackage{epsfig,graphics,mathrsfs}
\usepackage{graphicx}

\usepackage[usenames, dvipsnames]{color} 

\usepackage{hyperref}

\def \phi {\varphi}

\def \RN {\mathbb{R}^N}

\def \G{\Gamma}

\newcommand{\Rn}{\mathbb R^n}

\newcommand{\la}{\lambda}

\numberwithin{equation}{section}

\newcommand{\beq}{\begin{equation}}
\newcommand{\bea}[1]{\begin{array}{#1} }
\newcommand{\eeq}{ \end{equation}}
\newcommand{\ea}{ \end{array}}

\newcommand{\In}{\mathbf 1_E}
\newcommand{\Lp}{L^p}

\newtheorem{theorem}{Theorem}[section]

\newtheorem{proposition}[theorem]{Proposition}

\numberwithin{equation}{section}

\begin{document}

\title[On the best constant, etc.]{On the best constant in the nonlocal isoperimetric inequality of Almgren and Lieb}

\subjclass[2010]{}

\date{}

\begin{abstract}
In 1989 Almgren and Lieb proved a rearrangement inequality for the Sobolev spaces of fractional order $W^{s,p}$. The case $p = 2$ of their result implies the nonlocal isoperimetric inequality 
\[
\frac{P_s(E)}{|E|^{\frac{N-2s}N}} \ge \frac{P_s(B_1)}{|B_1|^{\frac{N-2s}N}},\ \ \ \ \ \ \ 0<s<1/2,
\]
where $P_s$ indicates the fractional $s$-perimeter, and $B_1$ is the unit ball in $\RN$.
In this note we explicitly compute the best constant, and  show that for any $0<s<1/2$, one has
\[
\frac{P_s(B_1)}{|B_1|^{\frac{N-2s}N}} = \frac{N \pi^{\frac N2 + s}  \G(1-2s)}{s \G(\frac N2+1)^{\frac{2s}N} \G(1-s)\G(\frac{N+2-2s}{2})}.
\] 
\end{abstract}

%\begin{abstract} 
%\end{abstract}
%\maketitle
\author{Nicola Garofalo}

\address{Dipartimento d'Ingegneria Civile e Ambientale (DICEA)\\ Universit\`a di Padova\\ Via Marzolo, 9 - 35131 Padova,  Italy}
\vskip 0.2in
\email{nicola.garofalo@unipd.it}

\thanks{The author was supported in part by a Progetto SID (Investimento Strategico di Dipartimento) ``Non-local operators in geometry and in free boundary problems, and their connection with the applied sciences", University of Padova, 2017.}

\maketitle

%\tableofcontents

\section{A simple proof of the computation of the best constant}\label{S:intro}

In their 1989 paper \cite[Theorem 9.2 (i)]{AL}, Almgren and Lieb proved that, if $f\in W^{s,p}$, for $0<s<1$ and $1\le p<\infty$, then also $f^\star\in W^{s,p}$ and 
\begin{equation}\label{AL}
||f^\star||_{W^{s,p}} \le ||f||_{W^{s,p}},
\end{equation}
where $f^\star$ denotes the non-increasing rearrangement of $|f|$. Here, for $1\le p < \infty$ and $s>0$ we have denoted by $W^{s,p}$ the Banach space of functions $f\in \Lp$ with finite Aronszajn-Gagliardo-Slobedetzky seminorm, 
\begin{equation}\label{ags}
[f]_{p,s} = \left(\int_{\RN} \int_{\RN} \frac{|f(x) - f(y)|^p}{|x-y|^{N+ps}} dx dy\right)^{1/p},
\end{equation}
see e.g. \cite{Ad} or also \cite{DPV} (throughout this note we assume $N\ge 2$). Notice that if $\delta_\la f(x) = f(\la x)$, with $\la>0$, then $[\delta_\la f]^p_{p,s} = \la^{-N+ps}[f]^p_{p,s}$. Consider now the nonlocal perimeter of a set,
 \begin{equation}\label{psE}
P_s(E) = [\mathbf 1_E]^2_{2,s} = [\In]_{1,2s}.
\end{equation} 
This notion has appeared in the works of Bourgain, Brezis and Mironescu \cite{BBM1}, \cite{BBM2}, \cite{B}, of Maz'ya \cite{Ma}, and of Caffarelli, Roquejoffre and Savin \cite{CRS}. These latter authors have begun the study of the Plateau problem with respect to a family of fractional perimeters. By the above noted scaling property, we have  
\begin{equation}\label{scale}
\frac{P_s(\delta_\la E)}{|\delta_\la E|^{(N-2s)/N}} = \frac{P_s(E)}{|E|^{(N-2s)/N}}.
\end{equation}
 This observation suggests that the fractional perimeter should satisfy the following isoperimetric inequality: given $0<s<1/2$, there exists a constant $i(N,s)>0$ such that for any measurable set $E\subset \RN$, such that $|E|<\infty$, one has
\begin{equation}\label{isos}
\frac{P_s(E)}{|E|^{(N-2s)/N}}\ \ge\ i(N,s).
\end{equation}
(we note that when $s\ge 1/2$ any non-empty open set has infinite $s$-perimeter, see e.g. the proof of Proposition \ref{P:AL} below). In fact, it is well-known (see e.g. \cite{FLS}, \cite{FS} and \cite{FFMMM}) that \eqref{isos} is contained in the Almgren-Lieb inequality \eqref{AL}, since the latter, combined with the observation \eqref{scale}, implies
\begin{equation}\label{AL20}
\frac{P_s(E)}{|E|^{\frac{N-2s}N}} \ge \frac{P_s(B_1)}{|B_1|^{\frac{N-2s}N}},
\end{equation}
where $B_1\subset \RN$ is the unit ball. The important case of equality in \eqref{AL20} is contained in the works \cite{FLS} and \cite{FS}. 

In this note we present a simple proof of the following explicit expression of the best constant in the right-hand side of \eqref{AL20}. In connection with our result the reader should see the remarks at the end of this note.

\begin{proposition}\label{P:AL}
For any $0<s<1/2$, one has 
\begin{equation}\label{isos2}
\frac{P_s(B_1)}{|B_1|^{\frac{N-2s}N}} =  \frac{N \pi^{\frac N2 + s}  \G(1-2s)}{s \G(\frac N2+1)^{\frac{2s}N} \G(1-s)\G(\frac{N+2-2s}{2})}.
\end{equation}
\end{proposition}

It is worth noting that the exact limiting behavior of the isoperimetric quotient $\frac{P_s(B_1)}{|B_1|^{\frac{N-2s}N}}$ at the poles $s = {\frac 12}$, or $s = 0$, is captured by the factor $\frac{\G(1-2s)}{s}$ (recall that $\G(z)$ has a simple pole at $z = 0$ with residue $1$).
Hereafter, we indicate with $\sigma_{N-1} = \frac{2 \pi^{\frac N2}}{\G(N/2)}$ the $(N-1)$-dimensional volume of the unit sphere $\mathbb S^{N-1}\subset \RN$, and with $\omega_N = \sigma_{N-1}/N$ the $N$-dimensional volume of the unit ball.
One has from \eqref{isos2}
\begin{equation}\label{0}
\underset{s\to 0^+}{\lim} s\ \frac{P_s(B_1)}{|B_1|^{\frac{N-2s}N}}  = \frac{2 \pi^{\frac N2}}{\G(\frac N2)} = \sigma_{N-1},\ \ \ \underset{s\to {\frac 12}^-}{\lim} (1 - 2s) P_s(B_1)  = \frac{2N \pi^{\frac N2}}{\G(\frac N2+1)^{\frac{1}N} \G(\frac{N+1}{2})} |B_1|^{\frac{N-1}N}.
\end{equation}
Both limit relations in \eqref{0} are special cases of well-known results. In fact, the case $p=2$ of \cite[Theor. 3]{MS} gives for $f\in W^{s,2}$,
\[
\underset{s\to 0^+}{\lim} s\ \int_{\RN}\int_{\RN} \frac{|f(x) - f(y)|^2}{|x-y|^{N+2s}} dx dy = \sigma_{N-1} ||f||^2_{L^2(\RN)}.
\] 
Taking $f = \mathbf 1_{B_1}$ in such result, we obtain the first relation in \eqref{0}. On the other hand, we recall that, in answer to a question posed in \cite{BBM1}, J. D\'avila in \cite[Theor. 1]{Davila} extended to any dimension their limiting formula for $N=1$, and proved  
\begin{equation}\label{dav}
\underset{s\nearrow 1/2}{\lim}\ (1 -2s) P_s(E) = \left(\int_{\mathbb S^{N-1}} |<e_N,\omega>|d\sigma(\omega)\right)\ P(E),
\end{equation}
where $e_N = (0,...,0,1)$. Since one easily recognises that $
\int_{\mathbb S^{N-1}} |<e_N,\omega>| = \frac{2\pi^{\frac{N-1}2}}{\G(\frac{N+1}2)}$, it is clear that taking $E = B_1$ in \eqref{dav} we obtain the latter relation in \eqref{0}.

\begin{proof}[Proof of Proposition \ref{P:AL}]
Using Plancherel theorem (we adopt the definition of Fourier transform $\hat f(\xi) = \int_{\RN} e^{-2\pi i<\xi,x>} f(x) dx$, which gives $||f||_2 = ||\hat f||_2$), we easily obtain for an arbitrary function $f\in L^2$
\[
 [f]^2_{2,s}  =  2 \int_{\RN} |\hat f(\xi)|^2 \int_{\RN} \frac{1 - \cos(2\pi<h,\xi>)}{|h|^{N+2s}} dh d\xi. 
\]
Now, a simple computation gives
\begin{align*}
& \int_{\RN} \frac{1 - \cos(2\pi<h,\xi>)}{|h|^{N+2s}} dh = (2 \pi |\xi|)^{2s} \int_{\RN} \frac{1 - \cos(h_N)}{|h|^{N+2s}} dh
= \frac{\pi^{\frac N2} \G(1-s)}{s 2^{2s} \G\left(\frac{N+ 2s}{2}\right)} (2 \pi |\xi|)^{2s},
\end{align*}
where in the last equality we have used the well-known identity
\[
\int_{\RN} \frac{1 - \cos(h_N)}{|h|^{N+2s}} dh =  \frac{ \pi^{\frac N2} \G(1-s)}{s 2^{2s} \G\left(\frac{N+ 2s}{2}\right)},
\]
see e.g. \cite[Lemma 3.1]{FLS}. We conclude that the fractional perimeter of the unit ball is given by
\begin{equation}\label{AL2}
P_s(B_1) = [\mathbf 1_{B_1}]_{2,s}^2 =  \frac{2\pi^{\frac N2 + 2s} \G(1-s)}{s \G\left(\frac{N+ 2s}{2}\right)} \int_{\RN} |\xi|^{2s} |\hat{\mathbf 1}_{B_1}(\xi)|^2   d\xi.
\end{equation}
In what follows, we denote by $J_\nu(z)$ the Bessel function of the first kind and order $\nu$. Using Bochner's formula 
$
\hat u(\xi)=2\pi|\xi|^{-\frac{N}2 +1}\int^\infty_0 r^{\frac{N}2} f(r)J_{\frac{N}2-1}
(2\pi|\xi|r) dr
$
for the Fourier transform of a spherically symmetric function $u(x) = f(|x|)$, see  \cite[Theorem 40 p.69]{BC}, in combination with the identity
$\int_0^1 x^{\nu+1} J_\nu(a x) dx = a^{-1} J_{\nu +1}(a),$ $\Re \nu > -1$,
see \cite[6.561, 5., p.683]{GR}, we have 
\begin{equation}\label{classical}
\hat{\mathbf 1}_{B_1}(\xi)  = 2\pi|\xi|^{-\frac{N}2 +1}\int^1_0 r^{\frac{N}2} J_{\frac{N}2-1}
(2\pi|\xi|r) dr = |\xi|^{-\frac N2} J_{\frac N2}(2\pi 
|\xi|).
\end{equation}
Since the asymptotic behaviour of $J_\nu$ is given by
$J_\nu(z)\cong\frac{2^{-\nu}}{\Gamma(\nu+1)}z^\nu$, as $z\to 0$, $J_\nu(z) = O(z^{-1/2})$, as $z\to+\infty$,
we see that $|\xi|^{s} \hat{\mathbf 1}_{B_1}(\xi) \in L^2(\RN)$ if and only if $s<1/2$ (notice that this shows that a ball has infinite $s$-perimeter if $1/2\le s<1$). For $0<s<1/2$ we thus find
\begin{equation}\label{opop}
\int_{\RN} |\xi|^{2s-N} |J_{\frac N2}(2\pi 
|\xi|)|^2 d\xi = \sigma_{N-1} \int_0^\infty r^{-(1-2s)} |J_{\frac N2}(2\pi 
r)|^2 dr.
\end{equation}
The latter integral can be computed explicitly using a special case of the beautiful, classical formula of Weber-Schafheitlin from 1880/1888:  let $\Re(\nu+\mu+1)>\Re \la >0$, $\alpha>0$, then
\begin{equation}\label{beauty}
\int_0^\infty r^{-\la} J_\nu(\alpha r) J_\mu(\alpha r) dr = \frac{\alpha^{\la -1} \G(\la)\G(\frac{\nu+\mu-\la+1}{2})}{2^\la \G(\frac{\mu-\nu+\la+1}{2})\G(\frac{\nu+\mu+\la+1}{2})\G(\frac{\nu-\mu+\la+1}{2})},
\end{equation}
see 6.574, 2. on p. 692 in \cite{GR}, but for a proof see 13.4 on p. 398 in \cite{Wa}, or the original papers of Sonine \cite[pp. 51-52]{So} and Schafheitlin \cite{Sch}. Taking $\nu = \mu = N/2$, $\la = 1-2s>0$, and $\alpha = 2\pi$ in \eqref{beauty}, we thus find
\[
\int_0^\infty r^{-(1-2s)} |J_{\frac N2}(2\pi 
r)|^2 dr = \frac{\G(1-2s)\G(\frac{N+2s}{2})}{2 \pi^{2s} \G(1-s)^2\G(\frac{N+2-2s}{2})}.
\] 
Combining this observation with \eqref{AL2}, \eqref{opop}, we obtain 
\begin{align*}
\frac{P_s(B_1)}{|B_1|^{\frac{N-2s}N}} & = \frac{2 \pi^{\frac N2 + 2s} \G(1-s)}{s \G\left(\frac{N+ 2s}{2}\right)} \omega_N^{\frac{2s-N}N} \sigma_{N-1} \frac{\G(1-2s)\G(\frac{N+2s}{2})}{2 \pi^{2s} \G(1-s)^2\G(\frac{N+2-2s}{2})}
\\
& = \frac{N \pi^{\frac N2 + s}  \G(1-2s)}{s \G(\frac N2+1)^{\frac{2s}N} \G(1-s)\G(\frac{N+2-2s}{2})},
\notag
\end{align*}
which is the desired conclusion \eqref{isos2}.

\end{proof}

In closing, the following two remarks are in order. First, in (4.2) and (1.4) of their 2008 work \cite{FS}, Frank and Seiringer had already shown that  
\begin{equation}\label{ins}
\frac{P_s(B_1)}{|B_1|^{\frac{N-2s}N}} = \frac{N \pi^s}{(N-2s)\G(\frac N2+1)^{\frac{2s}N}}\ C_{N,s,1},
\end{equation}
where (keeping in mind that their $s$ corresponds to our $2s$) they defined
\begin{equation}\label{C}
C_{N,s,1} = 2 \sigma_{N-2} \int_0^1 r^{-(1-2s)} (1 - r^{N-2s}) \int_{-1}^1 \frac{(1-t^2)^{\frac{N-3}2}}{(1-2rt + r^2)^{\frac{N+2s}2}} dt dr.
\end{equation}
The authors provide the explicit values of $C_{N,s,1}$ only for $N=1$ or $3$, but the integral in the right-hand side of \eqref{C} does not seem to be easily computable, in general. We note that our formula \eqref{isos2} does exactly that. 

Secondly, after a preliminary version of this note was completed, R. Frank has kindly informed us that the explicit value in our formula \eqref{isos2} can also be obtained by combining  Proposition 2.3 in the work \cite{FFMMM} with a result in Samko's book \cite{Sam} which is itself cited in \cite{FFMMM}. In a subsequent conversation, A. Figalli has kindly told us that, although an expression of the best constant is not explicitly written in their work, one can extract it from the  following chain of results (which for the reader's sake we have outlined in detail, also keeping in mind that their $s$ corresponds to our $2s$):
\begin{itemize}
\item[1)] The first key step is formula (2.11) in Proposition 2.3 in \cite{FFMMM} which states
\begin{equation}\label{lambda1s}
P_s(B_1) = \frac{\sigma_{N-1}}{2s(N-2s)} \la_1^s.
\end{equation}
Here, $\la_1^s$ indicates the first eigenvalue of the following operator in formula (2.7) in \cite{FFMMM} :
\[
\mathscr J_s u = \frac{2^{1-2s} \pi^{\frac{N-1}2} \G(\frac{1-2s}2)}{(1+2s) \G(\frac{N+2s}2)} \mathscr D^{1+s} u,
\]
where $\mathscr D^{1+s}$ is the hypersingular operator on $\mathbb S^{N-1}$ defined by
\[
\mathscr D^{1+s} u(x) = \frac{2s 2^s}{\pi^{\frac{N-1}2} } \frac{\G(\frac{N+2s}2)}{\G(\frac{1-2s}2)} \operatorname{P.V.} \int_{\mathbb S^{N-1}} \frac{u(x) - u(y)}{|x-y|^{N+2s}} d\sigma(y).
\]
\item[2)] Denoting by $\lambda_1^{\star}(s)$ the first eigenvalue of the operator $\mathscr D^{1+s}$, then by the above definition of $\mathscr J_s $ one has  that 
\[
\lambda_1^s = \frac{2^{1-2s} \pi^{\frac{N-1}2} \G(\frac{1-2s}2)}{(1+2s) \G(\frac{N+2s}2)}\ \lambda_1^{\star}(s).
\]
\item[3)] Finally, $\lambda_1^{\star}(s)$ is contained in Lemma 6.26 in \cite{Sam}. The latter gives (see also (2.4) in \cite{FFMMM})
\[
\lambda_1^{\star}(s) = \frac{\G(\frac{N+2+2s}2)}{\G(\frac{N-2s}2)} - \frac{\G(\frac{N+2s}2)}{\G(\frac{N- 2-2s}2)}. 
\] 
\end{itemize} 

However, it should be noted that the indirect proof outlined in (1)-(3) is not self-contained and relies on several auxiliary results. For instance, the proof of \eqref{lambda1s} above, i.e., (2.11) in \cite[Prop. 2.3]{FFMMM}, uses various special calculations involving the operator $\mathscr J_s$ and, per se, is at least as long as the whole proof of Proposition \ref{P:AL}. More importantly, (1)-(3) involve facts from harmonic analysis on the sphere $\mathbb S^{N-1}$ which our simple proof of \eqref{isos2} avoid altogether. For instance, it rests on Lemma 6.26 from \cite{Sam} which is not self-contained since its proof hinges on the Funk-Hecke formula for spherical harmonics (see \cite[Theor. 1.7]{Sam}), and on the expression, in terms of various special integrals involving Gegenbauer polynomials, of the coefficients in the Fourier-Laplace series of a function on the sphere. 

As a final comment, we note that since in Proposition \ref{P:AL} we explicitly compute $P_s(B_1)$, our result provides an alternative direct computation of the above mentioned number $\la_1^s$ in \eqref{lambda1s}.

%%%%%%%%%%%%%%%%%%%%%%%%%%%%%%%%%%%%%%%

\bibliographystyle{amsplain}

\end{document}